\documentclass{amsart}
\usepackage{amsfonts,amssymb,amscd,amsmath,enumerate,verbatim,newlfont,calc}
\usepackage{graphicx}

\newtheorem{theorem}{Theorem}[section]
\newtheorem{theorem A}{Theorem A}
\newtheorem{theorem B}{Theorem B}
\newtheorem{lemma}[theorem]{Lemma}

\newtheorem{definition}[theorem]{Definition}

\newtheorem{proposition}[theorem]{Proposition}
\begin{document}
\authors
\title{On the Schur multipliers of Lie superalgebras of maximal class}
\author[Z. Araghi Rostami]{Zeinab Araghi Rostami}
\author[P. Niroomand]{Peyman Niroomand}

\address{Department of Pure Mathematics\\
Ferdowsi University of Mashhad, Mashhad, Iran}
\email{araghirostami@gmail.com, zeinabaraghirostami@alumni.um.ac.ir}
\address{School of Mathematics and Computer Science\\
Damghan University, Damghan, Iran}
\email{niroomand@du.ac.ir, p$\_$niroomand@yahoo.com}

\address{Department of Mathematics, Ferdowsi University of Mashhad, Mashhad, Iran}
\keywords{Schur multiplier, Nilpotent Lie superalgebra, Maximal class.}
\thanks{\textit{Mathematics Subject Classification 2010.} 17B01, 17B05, 17B30, 19C09.}

\maketitle
\begin{abstract}

We categorize all non-abelian nilpotent Lie superalgebras of dimension $(m|n)$, where $1\leq s(L)\leq 10$, and $s(L)$ is a non-negative integer defined by Nayak. Furthermore, we classify the structure of all Lie superalgebras of dimension
at most five such that $\dim{L^2}=\dim\mathcal{M}(L)$.

\end{abstract}

\section{Introduction and preliminaries}
The theory of Lie superalgebras, first introduced by Kac in \cite{7}, has become a central framework in both theoretical physics and advanced algebra due to its profound connections to Lie groups and symmetry principles. This area has facilitated the study of supersymmetry, uniting bosonic and fermionic fields, and has provided a rich foundation for understanding algebraic structures in higher dimensions. Lie superalgebras have also been instrumental in addressing representation and classification problems, with many significant results emerging from this framework. These results, which expand upon classical Lie algebra theory, are surveyed comprehensively in \cite{1,4,7,7',8,9,10,13}.

One of the key characteristics of a Lie superalgebra is captured by examining the dimension of its Schur multiplier, an invariant that reveals critical structural information about nilpotent Lie algebras. For a non-zero integer \( t(L) \), the Schur multiplier's dimension is given by 
\[
\frac{1}{2}n(n-1) - t(L).
\] 
This invariant plays a pivotal role in characterizing nilpotent Lie algebras, with many authors having investigated this method for values of \( t(L) \) ranging from \( 0 \) to \( 8 \) as shown in \cite{2,5,6}. For Lie algebras of maximal class, classifications have been achieved for \( 0 \leq t(L) \leq 16 \), underscoring the depth and scope of these structural studies.

In addition to the Schur multiplier, another critical invariant in the study of Lie algebras is \( s(L) \), which measures the difference between the number of generators and relations. For a non-abelian nilpotent Lie algebra \( L \), \( s(L) \) is defined by
\[
s(L) = \frac{1}{2}(n-1)(n-2) + 1 - \dim \mathcal{M}(L),
\]
where \( s(L) > 0 \). This parameter has been pivotal in constructing the classifications of all nilpotent Lie algebras with values of \( s(L) \) between \( 0 \) and \( 3 \) (see \cite{11,12,14}). Additionally, in \cite{15}, the structure of Lie algebras of maximal class has been detailed for values of \( s(L) \) up to \( 15 \).

Extending these ideas, Nayak \cite{9} recently introduced bounds on the dimension of the Schur multipliers for Lie superalgebras, revealing that these multipliers function similarly to those of Lie algebras in characterizing finite-dimensional nilpotent structures. In this paper, we build on Nayak’s work by characterizing the structure of all maximal class Lie superalgebras \( L \) for \( 1 \leq s(L) \leq 10 \). Moreover, drawing parallels with the work in \cite{15}, we present the structure of all nilpotent Lie superalgebras of maximal class \( L \) where \( \dim \mathcal{M}(L) = \dim L^2 \). These results not only deepen the classification theory for Lie superalgebras but also provide tools to explore more complex structures that arise in higher-dimensional spaces and in applications to physics, particularly in supersymmetric theories.

All modules and algebras discussed in this paper are defined over a commutative ring \( \mathbb{K} \) with unity, ensuring a broad applicability to various algebraic settings. For consistency, we adhere to standard notations and definitions in Lie superalgebra theory, as outlined in \cite{4}. Let \( \mathbb{Z}_2 = \{ \bar{0}, \bar{1} \} \) be the grading field, where \( (-1)^{\bar{0}} = 1 \) and \( (-1)^{\bar{1}} = -1 \).

In this context, a \( \mathbb{Z}_2 \)-graded module (or supermodule) \( M \) decomposes into even and odd components, 
\[
M = M_{\bar{0}} \oplus M_{\bar{1}},
\]
with elements in \( M_{\bar{0}} \) and \( M_{\bar{1}} \) termed even and odd, respectively. Non-zero elements in \( M_{\bar{0}} \cup M_{\bar{1}} \) are termed homogeneous. For a homogeneous element \( m \in M_{\bar{\alpha}} \) with \( \alpha \in \mathbb{Z}_2 \), the degree of \( m \), denoted \( |m| = \bar{\alpha} \), signifies its parity. Thus, the notation \( |m| \) implies that \( m \) is a homogeneous element. A submodule \( N \subset M \) is called a \( \mathbb{Z}_2 \)-graded submodule (or subsupermodule) if it splits as 
$N = N_{\bar{0}} \oplus N_{\bar{1}},$
where \( N_{\bar{0}} = N \cap M_{\bar{0}} \) and \( N_{\bar{1}} = N \cap M_{\bar{1}} \).

\begin{definition}\cite[Definition 2.1]{4}\label{d2.1}
A Lie superalgebra is a superalgebra \( M = M_{\bar{0}} \oplus M_{\bar{1}} \) with a multiplication denoted by \( [ \cdot , \cdot ] \), known as the super bracket operation, which satisfies the following identities:

\begin{enumerate}
\renewcommand{\labelenumi}{(\roman{enumi})}
\item \( [x,y] = -(-1)^{|x||y|}[y,x] \),
\item \( [x,[y,z]] = [[x,y],z] + (-1)^{|x||y|}[y,[x,z]] \),
\item \( [m_{\bar{0}}, m_{\bar{0}}] = 0 \),
\end{enumerate}

for all homogeneous elements \( x, y, z \in M \) and \( m_{\bar{0}} \in M_{\bar{0}} \).
\end{definition}

\noindent It is noteworthy that the third equation follows from the first equation under the assumption that \( 2 \) is invertible in \( \mathbb{K} \). The second equation is equivalent to the graded Jacobi identity:
\[
(-1)^{|x||z|}[x,[y,z]] + (-1)^{|y||x|}[y,[z,x]] + (-1)^{|z||y|}[z,[x,y]] = 0.
\]

Applying these identities shows that for a Lie superalgebra \( M = M_{\bar{0}} \oplus M_{\bar{1}} \), the even part \( M_{\bar{0}} \) forms a Lie algebra, while the odd part \( M_{\bar{1}} \) is an \( M_{\bar{0}} \)-module. Consequently, if \( M_{\bar{1}} = 0 \), then \( M \) is a Lie algebra, whereas if \( M_{\bar{0}} = 0 \), \( M \) is an abelian Lie superalgebra, meaning \( [x,y] = 0 \) for all \( x, y \in M \). However, a general Lie superalgebra is not necessarily a Lie algebra.

A sub-superalgebra of \( L \) is a \( \mathbb{Z}_2 \)-graded vector subspace that is closed under the bracket operation. Taking the commutator of \( L \) with itself produces the graded subalgebra \( [L, L] \), denoted by \( L^2 \). A \( \mathbb{Z}_2 \)-graded subspace \( I \) is a graded ideal of \( L \) if \( [I, L] \subseteq I \). For any \( x \in L \), the ideal \( Z(L) = \{ z \in L \, | \, [z, x] = 0 \} \) is a graded ideal known as the center of \( L \). If \( I \) is an ideal of \( L \), the quotient Lie superalgebra \( L/I \) inherits a canonical Lie superalgebra structure, making the natural projection map a homomorphism. The notions of epimorphisms, isomorphisms, and automorphisms carry their usual meanings.

For a Lie superalgebra \( L = L_{\bar{0}} \oplus L_{\bar{1}} \) over a field, we describe it as an \( (m, n) \) Lie superalgebra if \( \dim L_{\bar{0}} = m \) and \( \dim L_{\bar{1}} = n \). Throughout this paper, \( A(m|n) \) denotes an abelian Lie superalgebra with dimension \( (m | n) \). The descending central sequence of a Lie superalgebra \( L \) is defined by \( L^{1} = L \) and \( L^{c+1} = [L^c, L] \) for all \( c \geq 1 \). If for some positive integer \( c \), \( L^{c+1} = 0 \) but \( L^c \neq 0 \), then \( L \) is called nilpotent with nilpotency class \( c \). Additionally, \( |[m, n]| = |m| + |n| \).

A homomorphism between superspaces \( f : V \to W \) of degree \( |f| \in \mathbb{Z}_2 \) is a linear map that satisfies \( f(V_{\bar{\alpha}}) \subseteq W_{\bar{\alpha} + |f|} \) for \( \alpha \in \mathbb{Z}_2 \). In particular, if \( |f| = \bar{0} \), then \( f \) is called a homogeneous linear map of even degree. A Lie superalgebra homomorphism \( f: M \to N \) is a homogeneous linear map of even degree such that \( f([x, y]) = [f(x), f(y)] \) for all \( x, y \in L \).
\\
\\
The following results relate to the Schur multiplier of a Lie superalgebra \( L \).

\begin{theorem}[11, Theorem 3.3]\label{t1.8}
Let \( L \) be a Lie superalgebra with \( \dim{L} = (m | n) \). Then \( \dim{\mathcal{M}(L)} \leq \frac{1}{2}[(n + m)^2 + (n - m)] \).
\end{theorem}

As a consequence, there exists a non-negative integer \( t(L) \) such that
\[
\dim{\mathcal{M}(L)} = \frac{1}{2}[(n + m)^2 + (n - m)] - t(L).
\]

\begin{theorem}[11, Theorem 3.4]\label{t1.9}
Let \( L \) be a Lie superalgebra with \( \dim{L} = (m | n) \). Then
\[
\dim{\mathcal{M}(L)} = \frac{1}{2}[(n + m)^2 + (n - m)]
\]
if and only if \( L \) is abelian.
\end{theorem}

Thus, \( t(L) = 0 \) if and only if \( L \) is abelian.

\begin{theorem}[11, Theorem 5.1]\label{t1.12}
Let \( L = L_{\bar{0}} \oplus L_{\bar{1}} \) be a nilpotent Lie superalgebra with \( \dim{L} = (m | n) \) and \( \dim{L^2} = r + s \) where \( r + s \geq 1 \). Then
\[
\dim{\mathcal{M}(L)} \leq \frac{1}{2}[(m + n + r + s - 2)(m + n - r - s - 1)] + n + 1.
\]
\end{theorem}

If \( r + s = 1 \), then equality holds if and only if \( L \cong H(1, 0) \oplus A(m - 3 | n) \), where \( A(m - 3 | n) \) is an abelian Lie superalgebra of dimension \( (m - 3 | n) \) and \( H(1, 0) \) is a special Heisenberg Lie superalgebra of dimension \( (3 | 0) \). Define \( s(L) \) as
\[
s(L) = \frac{1}{2}(m + n - 2)(m + n - 1) + n + 1 - \dim{\mathcal{M}(L)}.
\]
Then, \( \dim{\mathcal{M}(L)} = \frac{1}{2}(m + n - 2)(m + n - 1) + n + 1 - s(L) \) with \( s(L) \geq 0 \). Moreover, \( t(L) = m + n - 2 + s(L) \), suggesting that the classification of nilpotent Lie superalgebras \( L \) by \( s(L) \) correlates to a classification in terms of \( t(L) \).

\begin{theorem}\label{t1.14}
If \( L \) is a finite-dimensional nilpotent Lie superalgebra of dimension greater than \( 1 \) and class \( (p, q) \), then \( \mathcal{M}(L) \neq 0 \).
\end{theorem}

\begin{proof}
This result follows from [3, Theorem 3.2].
\end{proof}

\begin{definition}[15, Definition 4.1]\label{d1}
A Lie algebra \( L \) is called capable if \( L \cong H/Z(H) \) for some Lie algebra \( H \).
\end{definition}

We denote $Z^{*}(L)$ to be the smallest graded ideal in $L$ such that $L/Z^{*}(L)$ is capable, (see \cite{12'}).

\begin{lemma} [15, Lemma 4.3] A Lie superalgebra $L$ is capable if and only if $Z^{*}(L)=\{0\}$.
\end{lemma}

\begin{theorem} [15, Theorem 4.9]\label{tt}
Let $N$ be a central ideal in a Lie superalgebra $L$. Then the following conditions are equivalent
\begin{enumerate}[(i)]
\item{$\frac{\mathcal{M}(L/N)}{\mathcal{M}(L)}\cong N\cap L^2$}
\item{$N\subseteq Z^{*}(L)$}
\item{$\mathcal{M}(L)\to \mathcal{M}(L/N)$ is monomorphism.}
\end{enumerate}

\end{theorem}

\section{Main Results}
In this section, we undertake a comprehensive categorization of all non-abelian nilpotent Lie superalgebras of dimension $(m|n)$, constrained by the condition $1 \leq s(L) \leq 10$, where $s(L)$ is a non-negative integer as defined by Nayak. Our classification also encompasses all Lie superalgebras of dimension at most five for which $\dim L^2 = \dim \mathcal{M}(L)$, offering a structural perspective that refines previous frameworks and advances our understanding of these algebraic systems.
\begin{theorem}\label{t1.15}
Let $L$ be a nilpotent Lie superalgebra with dimensional greater than $2$ and $n\geq 1$ such that $\dim{L^2}=m+n-2$, then
$$(n+m)(n+m-3)+4\leq 2t(L) <(n+m)^2+n-m.$$
\end{theorem}
\begin{proof}
Using Theorem $5.4$ in \cite{9}, we have $\dim{\mathcal{M}(L)}\leq m+2n-2$. Since $\dim{\mathcal{M}(L)}=\frac{1}{2}[(n+m)^2+n-m]-t(L)$, we can see that
$$\frac{1}{2}[(n+m)^2+n-m]-t(L)\leq m+2n-2.$$
Thus, $(n+m)(n+m-3)+4\leq 2t(L)$. Now, Theorem \ref{t1.14} implies that $\dim{\mathcal{M}(L)}>0$, so $\frac{1}{2}[(n+m)^2+n-m]-t(L)>0$. Hence $2t(L)<(n+m)^2+n-m$ and the result follows.
\end{proof}

\begin{theorem}\label{t1.16}
Let $L$ be a nilpotent Lie superalgebra with dimensional greater than $2$ with $n\geq 1$ such that $\dim{L^2}=m+n-2$, then
$$(n+m)(n+m-5)+8\leq 2s(L) <(n+m)(n+m-1)-2m+4.$$
\end{theorem}

\begin{proof}
Since $t(L)=(n+m-2)+s(L)$, by using Theorem \ref{t1.15}, we have
$$(n+m)(n+m-3)+4\leq 2(n+m-2)+2s(L) <(n+m)^2+(n-m).$$
Thus,
\begin{align*}
(n+m)(n+m-3)+4-2(n+m-2) &\leq 2(n+m-2)+2s(L)\\
&<(n+m)^2+(n-m)-2(n+m-2)
\end{align*}
So we have $(n+m)(n+m-5)+8\leq 2s(L)$ and $2s(L) <(n+m)(n+m-1)-2m+4$ and the result is obtained.
\end{proof}

At present, we must classify nilpotent Lie superalgebras of dimension at most 5 that are not Lie algebras. N. Backhouse and N. L. Matiadou classified these Lie superalgebras into two categories: trivial and non-trivial. It is noteworthy that the Lie superalgebra $L$ is considered trivial if $[L_{\bar{1}},L_{\bar{1}}]=0$, and non-trivial otherwise. 
Following the notations in \cite{1}, we denote the elements of $L_{\bar{0}}$ $($resp $L_{\bar{1}})$ by Latin letters $($resp Greek letters$)$ taken from the beginning of the alphabet. Using this classification, we have the following Lie superalgebras with maximal class $(\dim L^2=m+n-2)$. In the following section, we will compute the Schur multiplier of these in Theorem \ref{t}.

\begin{tabular}{llc}
\\
{Table $1$, $(1,2)$-Lie superalgebra}
\\
\hline
Name&Relations&$\dim{\mathcal{M}(L)}$\\
\hline
Trivial LS\\
$L_{1,2}^{(3)}$ & $[\alpha,\beta]=\alpha$&$2$\\
Non-Trivial LS\\
$L_{1,2}^{(1)}$ & $[\alpha,\alpha]=a , [\beta,\beta]=a$&$2$\\
$L_{1,2}^{(2)}$ & $[\alpha,\alpha]=a , [\beta,\beta]=-a$&$2$\\
\hline
\end{tabular}

\begin{tabular}{llc}
\\
{Table $2$, $(1,3)$-Lie superalgebra}
\\
\hline
Name&Relations&$\dim{\mathcal{M}(L)}$\\
\hline
Trivial LS\\

$L_{1,3}^{(5)}$ & $[a,\beta]=\alpha , [a,\gamma]=\beta$&$3$\\
\hline
\end{tabular}

\begin{tabular}{llc}
\\
{Table $3$, $(2,2)$-Lie superalgebra}
\\
\hline
Name&Relations&$\dim{\mathcal{M}(L)}$\\
\hline
Non-Trivial LS\\
$L_{2,2}^{(9)}$ & $[\alpha,\alpha]=a , [\beta,\beta]=b$&$2$\\
$L_{2,2}^{(10)}$ & $[\alpha,\alpha]=a , [\beta,\beta]=b, [\alpha,\beta]=a$&$1$\\
$L_{2,2}^{(11)}$ & $[\alpha,\alpha]=a, [\beta,\beta]=b,$&$1$\\
$\ $&$[\alpha,\beta]=p(a+b) \ ; \  p>0$&$\ $\\
$L_{2,2}^{(12)}$ & $[\alpha,\alpha]=a, [\beta,\beta]=b,$&$1$\\
$\ $&$[\alpha,\beta]=p(a-b) \ ; \  p>0$&$\ $\\

\hline
\end{tabular}

\begin{tabular}{llc}
\\
{Table $4$, $(1,4)$-Lie superalgebra}
\\
\hline
Name&Relations&$\dim{\mathcal{M}(L)}$\\
\hline
Trivial LS\\
$E^{22}$ & $[a,\alpha]=\beta, [a,\beta]=\gamma, [a,\gamma]=\delta,$&$6$\\

\hline
\end{tabular}

\begin{tabular}{llc}
\\
{Table $5$, $(3,2)$-Lie superalgebra}
\\
\hline
Name&Relations&$\dim{\mathcal{M}(L)}$\\
\hline
Non-Trivial LS\\

$3A_{1,1}+2A$ & $[\alpha,\alpha]=a, [\beta,\beta]=b, [\alpha,\beta]=c$&$3$\\
\hline
\end{tabular}

\begin{tabular}{llc}
\\
{Table $6$, $(2,3)$-Lie superalgebras}
\\
\hline
Name&Relations&$\dim{\mathcal{M}(L)}$\\
\hline
Non-Trivial LS\\

$(D^{15}+A_{1,1})^1$ & $[a,\beta]=\alpha, [a,\gamma]=\beta, [\gamma,\gamma]=b$&$1$\\

$(D^{15}+A_{1,1})^2$ & $[a,\beta]=\alpha, [a,\gamma]=\beta, [\beta,\beta]=b$&$2$\\
$\ $&$[a,\gamma]=-b$&$\ $\\

$(D^{15}+A_{1,1})^3$ & $[a,\beta]=\alpha, [a,\gamma]=\beta, [\beta,\beta]=b$&$3$\\
$\ $&$[\gamma,\gamma]=b, [\alpha,\gamma]=-b$&$\ $\\

$(D^{15}+A_{1,1})^4$ & $[a,\beta]=\alpha, [a,\gamma]=\beta, [\beta,\beta]=b$&$2$\\
$\ $&$[\gamma,\gamma]=-b, [\alpha,\gamma]=-b$&$\ $\\
\hline
\\
\end{tabular}

According to the free presentations in the previous tables, we want to compute the Schur multiplier of them.
\begin{theorem}\label{t}
The Schur multiplier of Lie superalgebras is given in the Table $7$.
\begin{tabular}{lr|lr}
\\
{Table $7$, Schur Multiplier}
\\
\hline
LS&$\mathcal{M}(L)$&LS&$\mathcal{M}(L)$\\
\hline
$L_{1,2}^{(1)}$ &$A(2|0)$&$(D^{15}+A_{1,1}))^3$&$A(2|1)$\\
\hline
$L_{1,2}^{(2)}$&$A(2|0)$&$L_{2,2}^{(10)}$&$A(1|0)$\\
\hline
$L_{1,2}^{(3)}$&$A(1|1)$&$L_{2,2}^{(11)}$&$A(1|0)$\\
\hline
$L_{2,2}^{(9)}$&$A(1|1)$&$L_{2,2}^{(12)}$&$A(1|0)$\\
\hline
$(D^{15}+A_{1,1})^2$&$A(1|1)$&$E^{22}$&$A(5|1)$\\
\hline
$(D^{15}+A_{1,1})^4$&$A(1|1)$&$3A_{1,1}+2A$&$A(1|2)$\\
\hline
$L_{1,3}^{(5)}$&$A(2|1)$&$(D^{15}+A_{1,1})^1$&$A(0|1)$\\

\end{tabular}
\\
\end{theorem}
\begin{proof}
We provide a comprehensive proof for $L_{1,2}^{(3)}$ and $(D^{15}+A_{1,1})^4$. The remaining cases will be proven in a manner analogous to this. 
Let $L \cong L_{1,2}^{(3)}$, as defined in  \cite[Definition 1.3]{1'}. According to this definition, $\alpha \wedge \alpha=\alpha \wedge \beta=0$. Therefore, $L\wedge L=<a\wedge \alpha , a\wedge \beta , \beta \wedge \beta >$.
Consequently, for all $w \in \mathcal{M}(L)$, there exists $\alpha_1 , \alpha_2 , \alpha_3\in \Bbb{R}$, such that
$w={\alpha_1}(a\wedge \alpha)+{\alpha_2}(a\wedge \beta)+{\alpha_3}(\beta \wedge \beta)$.

Now let $\tilde{\kappa}: L\wedge L\to [L,L]$ be given by $(x\wedge y \to [x,y])$. Since $\tilde{\kappa}(w)=0$, we have ${\alpha_1}[a,\alpha]+{\alpha_2}[a, \beta]+{\alpha_3}[\beta , \beta]$, so ${\alpha_2}{\alpha}=0$ and $\alpha_2=0$. Thus $w={\alpha_1}(a\wedge \alpha)+{\alpha_3}(\beta \wedge \beta)$. Therefore $\mathcal{M}(L)=<a\wedge \alpha , \alpha \wedge \alpha>$ and $\dim \mathcal{M}(L)=2$. On the other hand, by using \cite[Definition 1.3]{1'}, $\tilde{\kappa}$ is a homogeneous linear map of even degree. Also, we know
$|[a,\alpha]|=|a|+|\alpha |=0+1=1$ and $|[\beta,\beta]|=|\beta |+|\beta |=1+1=0$. Hence $a\wedge \alpha$ and $\beta \wedge \beta$ are odd and even elements of $L\wedge L$, respectively. Finally, since the Schur multiplier of a Lie superalgebra is abelian, we have $\mathcal{M}(L)\cong A(1|1)$.
\\
Let $L\cong (D^{15}+A_{1,1})^4$, by using the \cite[Definition 1.3]{1'}, we have $a\wedge b=b\wedge \alpha =b\wedge \beta =b\wedge \gamma = \alpha\wedge \alpha =\alpha \wedge \beta =\beta \wedge \gamma =0,$ and $\beta\wedge \beta =-\gamma \wedge \alpha$. Thus $L\wedge L=<a\wedge \alpha , a\wedge \beta , a\wedge \gamma , \alpha \wedge \alpha , \gamma \wedge \gamma >$.
Hence for all $w\in \mathcal{M}(L)$, there exist $\alpha_1 , \alpha_2 , \alpha_3 , \alpha_4 , \alpha_5 \in \Bbb{R}$, such that
$w={\alpha_1}(a\wedge \alpha)+{\alpha_2}(a\wedge \beta)+{\alpha_3}(a\wedge \gamma)+{\alpha_4}(\alpha \wedge \alpha)+{\alpha_5}(\gamma \wedge \gamma)$.
Now let $\tilde{\kappa}: L\wedge L\to [L,L]$ be given by $(x\wedge y \to [x,y])$. Since $\tilde{\kappa}(w)=0$, we have ${\alpha_1}[a,\alpha]+{\alpha_2}[a, \beta]+{\alpha_3}[a, \gamma]+{\alpha_4}[\alpha, \alpha]+{\alpha_5}[\gamma, \gamma]=0$, so, ${-\alpha_1}{\beta}+{\alpha_2}{\alpha}+{(\alpha_4+\alpha_5)}{b}=0$ and $\alpha_1, \alpha_2=0, \alpha_5=-\alpha_4$. Thus $w={\alpha_3}(a\wedge \gamma)+{\alpha_4}(\alpha \wedge \alpha - \gamma \wedge \gamma)$. Therefore $\mathcal{M}(L)=<a\wedge \gamma , \alpha \wedge \alpha - \gamma \wedge \gamma>$ and $\dim \mathcal{M}(L)=2$. On the other hand, by using Definition $1.3$ in \cite{1'}, $\tilde{\kappa}$ is a homogeneous linear map of even degree. Also, we know
$|[a,\gamma]|=1, |[\alpha,\alpha]|=0$ and $|[\gamma,\gamma]|=0$. Hence $a\wedge \gamma$, $\alpha \wedge \alpha - \gamma \wedge \gamma$ are odd and even elements of $L\wedge L$, respectively. Finally, since the Schur multiplier of a Lie superalgebra is abelian, we have $\mathcal{M}(L)\cong A(1|1)$.
\end{proof}
At present, our focus is on classifying non-abelian nilpotent Lie superalgebras of maximal class for $1\leq s(L)\leq 10$.
\begin{theorem}\label{t3}
Let $L$ be a non-abelian $(m|n)$-dimensional nilpotent Lie superalgebra of maximal class and $m+n\geq 3$. Then $1\leq s(L)\leq 10$ if and only if $L$ is isomorphic to the one of the Lie superalgebras listed in the table $8$.
\end{theorem}

\begin{tabular}{ll}
{Table $8$}
\\
\hline
$s(L)$&$\text{Name}$\\
\hline
$1$&$\text{There is no LS}$\\
\hline
$2$&$L_{1,2}^{(1)}, L_{1,2}^{(2)}, L_{1,2}^{(3)}$\\
\hline
$3$&$\text{There is no LS}$\\
\hline
$4$&$L_{1,3}^{(5)}, L_{2,2}^{(9)}$\\
\hline
$5$&$L_{2,2}^{(10)}, L_{2,2}^{(11)}, L_{2,2}^{(12)}, E^{22},$\\
\hline
$6$&$3A_{1,1}+2A$\\
\hline
$7$&$(D^{15}+A_{1,1})^3$\\

\hline
$8$&$(D^{15}+A_{1,1})^2, (D^{15}+A_{1,1})^4$\\

\hline
$9$&$(D^{15}+A_{1,1})^1$\\
\hline
$10$&$\text{There is no such Lie superalgebra}$\\
\hline

\end{tabular}
\\
\\
\begin{proof}
Utilizing Theorem \ref{t1.16}, if $s(L)=1$ then $m+n=3$ and $\dim{\mathcal{M}(L)}=3$, so by examining Table $1$, there is no such Lie superalgebra with these properties.
Let $s(L)=$2, as indicated in Tables $1, 2$, and $3$ if $m+n=3$, then we should have $L \cong L_{1,2}^{(1)}$, $L_{1,2}^{(2)}$, and $L \cong L_{1,2}^{(3)}$. Additionally when $m+n=4$ exists no Lie superalgebra that satisfies our conditions.

Given that $s(L)=3$, examining Tables $1, 2, 3$, we observe that for $m+n=3$ or $4$, there is no such Lie superalgebra.
If $s(L)=4$, then $m+n=4$. By referring to Tables $2$ and $3$, we have
$L\cong L_{1,3}^{(5)}$ or $L_{2,2}^{(9)}$.
For $s(L)\leq 5$, we have a classification of non-abelian nilpotent Lie superalgebras of maximal class and dimension at most $5$. Similarly, for $5\leq s(L)\leq 10$, we have a classification of non-abelian nilpotent Lie superalgebras of maximal class and dimension at most $5$. We present the summary of the results in Table $8$.
\end{proof}

At present, our objective is to elucidate the structural characteristics of all  $(m|n)$-dimensional nilpotent Lie superalgebras $L$ when $\dim L^2 = \dim \mathcal{M}(L) = m+n - 2$ for all $m+n \leq 5$.

\begin{proposition} \label{p1}
Let $L$ be an $(m|n)$-dimensional nilpotent Lie superalgebra such that $\dim L^2 = \dim \mathcal{M}(L) = m+n - 2$ and $m+n \leq 5$. Then $L$ is isomorphic to one of
the nilpotent Lie superalgebras $L_{2,2}^{(9)}$, $3A_{1,1}+2A$ or $(D^{15}+A_{1,1})^3$.
\end{proposition}
\begin{proof}
By examining the classification of all nilpotent Lie superalgebras of dimension at most five presented in \cite{1,7'}, we observe that only Lie superalgebras whose derived superalgebra possesses dimension $m+n-2$ are included in Tables $1$ to $6$. Utilizing a similar methodology as employed in the proof of Theorem \ref{t}, 
we have $\dim{\mathcal{M}(L_{2,2}^{(9)})}=2$, $\dim{\mathcal{M}(3A_{1,1}+2A)}=3$ and $\dim{\mathcal{M}((D^{15}+A_{1,1})^3)}=3$. Hence $L$ should be isomorphic to one of the nilpotent Lie superalges $L_{2,2}^{(9)}$, $3A_{1,1}+2A$ or $(D^{15}+A_{1,1})^3$.
\end{proof}

\begin{proposition} Let $L$ be an $(m|n)$-dimensional nilpotent Lie algebra of maximal class. If $\dim \mathcal{M}(L) = \dim L^2$, then $L$ is capable.

\end{proposition}

\begin{proof}
According to Proposition \ref{p1}, $L$ is isomorphic to one of the Lie superalgebras $L_{2,2}^{(9)}$, $3A_{1,1}+2A$, and $(D^{15}+A_{1,1})^3$. Now, let us assume that $L\cong 3A_{1,1}+2A$ is non capable. By applying Theorem \ref{tt}, we can find a non-zero element $x\in Z^*(L)$. We consider two cases:

Case $1$: If $x$ is an even element of $Z^*(L)$, then according to Table $2$, $L$ is isomorphic to one of the $(2|2)$-dimensional Lie superalgebras $L_{2,2}^{(9)}$, $L_{2,2}^{(10)}$, $L_{2,2}^{(11)}$, or $L_{2,2}^{(12)}$. In these cases, the dimension of their Schur multipliers is either $1$ or $2$. Therefore, the map $\mathcal{M}(L)\to \mathcal{M}(\frac{L}{<x>})$ is not injective, which contradicts part $(iii)$ of Theorem \ref{tt}. Hence, in this case, $L$ is capable.

Case $2$: If $x$ is an odd element of $Z^*(L)$, then by Black house’s classification in \cite{1}, $L\cong L_{3,1}^{(1)}=<a,b,c,\alpha \ | \ [b,c]=a , [b,\alpha ]=\alpha>$ and similar to Theorem \ref{t}, it can be demonstrated that $\dim \mathcal{M}(L_{1,3}^{(1)})=2$. Consequently, the map $\mathcal{M}(L)\to \mathcal{M}(\frac{L}{<x>})$ is not injective, leading to a contradiction. Hence, in this scenario, $L$ possesses the capability. It is also readily apparent that  $L_{2,2}^{(9)}$ and $(D^{15}+A_{1,1})^3$ are capable.

\end{proof}

\end{document}